\documentclass[10pt]{amsart}
\usepackage{amsmath}
\usepackage{amssymb}
\usepackage{graphicx}

\usepackage{geometry}
\usepackage[usenames,dvipsnames]{xcolor}

\def\beq*{\begin{eqnarray*}}
\def\eeq*{\end{eqnarray*}}
\def\be{\begin{eqnarray}}
\def\ee{\end{eqnarray}}
\def\b{\beta}
\def\hb{\widehat{\beta}}
\def\a{\alpha}
\def\ha{\widehat{\alpha}}

\def\hk{\widehat{k}}

\def\haa{\widehat{a}}
\def\hbb{\widehat{b}}

\def\mcS{\mathcal{S}}

\def\mbP{\mathbb{P}}

\newtheorem{theorem}{Theorem}
\newtheorem{lemma}[theorem]{Lemma}
\newtheorem{corollary}[theorem]{Corollary}

\newtheorem{definition}{Definition}

\newtheorem{case}{Case}
\newtheorem{subcase}{Case}
\numberwithin{subcase}{case}

\author[Amanda Lohss]{Amanda Lohss}
\address{Department of Mathematics, Drexel University, Philadelphia, 
PA  19104, USA} 
\email{agp47@drexel.edu}

\title[Asymptotic  distribution of
 symbols on diagonals]{The Asymptotic  distribution of
 symbols on diagonals
\\ 
of random weighted staircase tableaux
}

\begin{document}
\maketitle
\begin{abstract}
Staircase tableaux are combinatorial objects that were first introduced due to a connection with the asymmetric simple exclusion process (ASEP) and Askey-Wilson polynomials. Since their introduction, staircase tableaux have been the object of study in many recent papers. Relevant to this paper, the distribution of parameters on the first diagonal was proven to be asymptotically normal. In that same paper, a conjecture was made that the other diagonals would be asymptotically Poisson. Since then, only the second and the third diagonal were proven to follow the conjecture. This paper builds upon those results to prove the conjecture for fixed $k$. In particular, we prove that the distribution of the number of $\a$'s ($\b$'s) on the $k$th diagonal, $k>1$, is asymptotically Poisson with parameter $1/2$. In addition, we prove that symbols on the $k$th diagonal are asymptotically independent and thus, collectively follow the Poisson distribution with parameter $1$.
\end{abstract}

%
%
%
\section{Introduction}\label{intro}
This paper is on staircase tableaux, combinatorical objects which were introduced in \cite{CW1} and \cite{CW2}. Staircase tableaux were initially introduced to provide a combinatorical formula for the stationary distribution of the asymmetric simple exclusion process (ASEP), (see \cite[Theorem~2]{CW1} and \cite[Theorem 3.5]{CW2}). The ASEP is an important model from statistical mechanics that has been used to model protein synthesis \cite{GMP}, traffic flow \cite{SW2}, and much more (See, for example \cite{Bu}, \cite{DLS}). We refer to \cite[Section 2-3]{CW2} for more details on the ASEP and its connection to staircase tableaux. \\
Staircase tableaux also have interesting connections in combinatorics and analysis. In the original papers (\cite{CW1} and \cite{CW2}), staircase tableaux were used to compute the moments of Askey-Wilson polynomials (\cite[Theorem 3]{CW1}), an important family of orthogonal polynomials. In addition, staircase tableaux are in bijection with permutation tableaux, alternative tableaux, and tree-like tableaux (see \cite[Table~1]{HJ} for more details).\\
These connections have been the motivation for many recent papers to focus specifically on staircase tableaux (\cite{CD-H}, \cite{CSSW}, \cite{D-HH}, \cite{HJ}, \cite{HL}, and \cite{HP}). Relevant to this paper, the distribution of parameters along the first diagonal of staircase tableaux was proven to be asymptotically normal in \cite{HJ}. In addition, the expected number of parameters on any diagonal was computed which prompted a conjecture that the distribution of the $k$th diagonal is asymptotically Poisson as $n,k\rightarrow\infty$. This conjecture was partially proven in \cite{HL}. More specifically, parameters on the second and third diagonal were proven to be asymptotically Poisson. This paper builds upon those results to prove the conjecture for fixed $k$. In particular, we prove that the distribution of the number of $\a$'s ($\b$'s) on the $k$th diagonal is asymptotically Poisson with parameter $1/2$ (Theorem~\ref{alpha_on_kth} (Corollary~\ref{beta_on_kth})). In addition, we prove that symbols on the $k$th diagonal are asymptotically independent and thus, collectively follow the Poisson distribution with parameter $1$ (Theorem~\ref{symbol_on_kth}). \\

\section{Definitions, Notation, and Previous Results}

\begin{definition}\label{def} Staircase tableaux of size $n$ are defined to be a Young diagram of shape $n, n-1,\dots,1$ whose boxes are either empty or filled with an $\a, \b, \gamma$ or $\delta$. The filling rules are as follows:
\begin{itemize}
\item Each box along the first diagonal must contain a symbol.
\item All boxes in the same column and north of an $\a$ or $\gamma$ must be empty.
\item All boxes in the same row and west of a $\b$ or $\delta$ must be empty.
\end{itemize}
\end{definition}
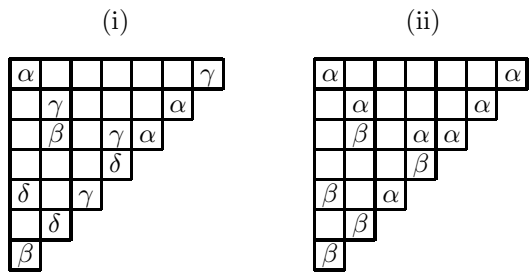
\begin{figure}[htbp] 
\setlength{\unitlength}{0.4cm}
\begin{center}
\begin{picture}
(0,0)(8,8)\thicklines

\put(3,8){(i)} \put(13,8){(ii)}


\put(0,0){\line(0,1){7}}
\put(1,0){\line(0,1){7}}
\put(2,1){\line(0,1){6}}
\put(3,2){\line(0,1){5}}
\put(4,3){\line(0,1){4}}
\put(5,4){\line(0,1){3}}
\put(6,5){\line(0,1){2}}
\put(7,6){\line(0,1){1}}

\put(0,7){\line(1,0){7}}
\put(0,6){\line(1,0){7}}
\put(0,5){\line(1,0){6}}
\put(0,4){\line(1,0){5}}
\put(0,3){\line(1,0){4}}
\put(0,2){\line(1,0){3}}
\put(0,1){\line(1,0){2}}
\put(0,0){\line(1,0){1}}

\put(0.25, 6.25){$\a$}
\put(0.25, 2.25){$\delta$}
\put(1.25, 5.25){$\gamma$}
\put(1.25, 4.25){$\b$}
\put(3.25, 4.25){$\gamma$}
\put(6.25, 6.25){$\gamma$}
\put(5.25, 5.25){$\alpha$}
\put(4.25, 4.25){$\alpha$}
\put(3.25, 3.25){$\delta$}
\put(2.25, 2.25){$\gamma$}
\put(1.25, 1.25){$\delta$}
\put(0.25, 0.25){$\b$}

\put(10,0){\line(0,1){7}}
\put(11,0){\line(0,1){7}}
\put(12,1){\line(0,1){6}}
\put(13,2){\line(0,1){5}}
\put(14,3){\line(0,1){4}}
\put(15,4){\line(0,1){3}}
\put(16,5){\line(0,1){2}}
\put(17,6){\line(0,1){1}}

\put(10,7){\line(1,0){7}}
\put(10,6){\line(1,0){7}}
\put(10,5){\line(1,0){6}}
\put(10,4){\line(1,0){5}}
\put(10,3){\line(1,0){4}}
\put(10,2){\line(1,0){3}}
\put(10,1){\line(1,0){2}}
\put(10,0){\line(1,0){1}}

\put(10.25, 6.25){$\a$}
\put(10.25, 2.25){$\b$}
\put(11.25, 5.25){$\a$}
\put(11.25, 4.25){$\b$}
\put(13.25, 4.25){$\a$}
\put(16.25, 6.25){$\a$}
\put(15.25, 5.25){$\a$}
\put(14.25, 4.25){$\a$}
\put(13.25, 3.25){$\b$}
\put(12.25, 2.25){$\a$}
\put(11.25, 1.25){$\b$}
\put(10.25, 0.25){$\b$}

\end{picture}

\vspace{3cm}
\caption{A staircase tableau of size $7$ with weight $\a^{3}
  \b^{2} \delta^{3} \gamma^{4}$. (ii) The $\a/\b$-staircase tableau version of (i) obtained by replacing $\gamma$'s with $\a$'s and $\delta$'s with $\b$'s.}
\end{center}

\end{figure}
We enumerate the rows and columns of staircase tableaux beginning in the NW-corner. Let $(i,j)$ refer to the box in the $i$th row and the $j$th column. We are especially interested in the diagonals of staircase tableaux and more formally define the $k$th diagonal to be the collection of boxes $(i,j)$ such that $i+j=n-k+2$. It is important to note that the $k$th diagonal in this paper is the $(n-k+1)$th diagonal in \cite{HJ}. It is more convenient to enumerate the diagonals in this way as we will be considering $k$th diagonal where $k$ is finite instead of the $(n-k+1)$th diagonal where $n-k+1$ is finite. We also find it convenient to enumerate diagonal boxes by their column position i.e. we write $\{(n-k-l+2,l):\ 1\le l\le n-k+1\}$ for the boxes on the $k^{th}$ diagonal. Throughout the paper, $x_j^k$ will denote the event (or its indicator) that there is a symbol $x$ on the $k^{th}$ diagonal in the $j^{th}$ column of a staircase tableau; that is, there is an $x$ in the $(n-k-j_1+2,j_1)$ box of the tableau.\\
Without loss of generality, we can replace all $\gamma$'s with $\a$'s and $\delta$'s with $\b$'s to obtain $\a/\b$-staircase tableaux. Let $S_n$ denote the set of all $\a/\b$-staircase tableaux. The weight generating function for staircase tableaux is well-known and the following version for $S_n$ is from \cite{HJ},
\[
Z_{n}(\a, \b):=\sum_{S \in \mcS_{n}} wt(S)= \a^{n}\b^{n}(a + b+n-1)_{n}
\]
where $wt(S)$ is the product of all symbols in $S$, $a:= \a^{-1}$, $b:= \b^{-1}$  and $(x)_n=x(x-1)\dots(x-(n-1))$ is the falling factorial of $x$.\\
It is important to note that there is a simple involution between staircase tableaux of a particular size. If we interchange the rows and columns and $\a$'s and $\b$'s, we obtain a staircase tableaux of the same size. We will use this involution to obtain the distribution of $\b$'s from the distribution of $\a$'s.\\
We want to consider random weighted $\a/\b$-staircase tableaux as was done in \cite{HJ}, 
\begin{definition}\label{def1}
For all $n \geq 1$, $\a, \b \in [0, \infty)$ with $(\a,
\b) \neq (0,0)$, define a random $S\in \mcS_n$ to be the $\a/\b$-staircase tableaux chosen according to the following distribution,
\[
\mathbb{P}_{n,\a,\b} (S) := \frac{wt(S)}{Z_{n}(\a,\b)},\quad S \in \mcS.
\]
\end{definition}
We allow $\a\rightarrow\infty$ ($\b\rightarrow\infty$) by fixing $\b$ ($\a$) and taking the limit, or we also allow $\a=\b\rightarrow\infty$. Even though we are using a probability measure on $S_n$, all of our results can be extended to tableaux as defined in Definition~\ref{def}, i.e. staircase tableaux with all four symbols. This can be done by randomly replacing each $\a$ with $\gamma$ with probability $\frac{\gamma}{\a + \gamma}$ and each $\b$ with $\delta$ with probability $\frac{\delta}{\b + \delta}$ independently for each occurrence.\\
Using the probability measure from Definition~\ref{def1}, the distribution of boxes was given in \cite{HJ}. For boxes on the first diagonal (\cite[Theorem~7.1]{HJ}),
\be
\mathbb{P}_{n, \a, \b}(\a^1_j) = \frac{k+j+b-2}{n + a + b - 1},\quad \mathbb{P}_{n, \a, \b}(\b^1_j) = \frac{n-k-j+a+1}{n + a + b - 1}\label{subt}.
\ee
For boxes on the $k$th diagonal, $k>1$ (\cite[Theorem~7.2]{HJ}),
\be
\mathbb{P}_{n, \a, \b}(\a^k_j) = \frac{b + j - 1}{(n-k + a + b + 1)_2}, 
\qquad\mathbb{P}_{n, \a, \b}(\b^k_j)  = \frac{n-k-j+a+1}{(n-k + a + b + 1)_2}.
\label{boxk}
\ee
Of course boxes on the $k$th diagonal may be empty if $k>1$, and this probability is the remainder of the above.\\
One more useful result from \cite{HJ} relates to subtableaux. For arbitrary $S\in\mcS_n$, let $S[i,j]$ be the subtableau of $S$ that begins in box $(i,j)$, i.e. the first $i-1$ rows and $j-1$ columns of $S$ are deleted. Then, as proven in \cite[Theorem~6.1]{HJ},
\be
\mbP_{n,\a,\b}(S)=\mbP_{n-i-j+2, \ha, \hb}(S[i,j])\label{subtab}
\ee
where $\haa=a+i-1$ and $\hbb=b+j-1$.\\
The distributions of the second and third diagonals were shown to converge to a Poisson random variable in \cite{HL} by the method of factorial moments. We will also use this method for the $k$th diagonal in general, and the following two lemmas proven in \cite{HL} will be useful (\cite[Lemma 1]{HL} and \cite[Lemma 5]{HL}, respectively).
\begin{lemma}\label{fac_mom} Let $Y=\sum_{j=1}^mI_j$, where $(I_j)$ are indicator random variables.  Then, for $r\ge1$, 
\[\mathbb{E}(Y)_{r} =
r! \left( \sum_{1 \leq j_{1} < ... < j_{r} \leq m } \mathbb{P}(I_{j_{1}} \cap \ldots \cap I_{j_{r}}) \right), 
\]
where $\mathbb E(X)_r=\mathbb EX(X-1)\dots(X-(r-1))$ is the $r^{th}$  factorial moment.
\end{lemma}
\begin{lemma}
\label{LA}
Let 
\[J_{r,m}:=\{ 1 \leq j_{1} < ... < j_{r} \leq m:\ j_{k} \leq j_{k + 1}
- 2, \hspace{2mm} \forall k = 1, 2, ..., r - 1\}.\] 
Then 
\[\sum_{J_{r, m}} \left(\prod_{k=1}^{r} j_{r-k+1}\right)
=\frac{(m+1)_{2r}}{2^{r}r!}.\]
\end{lemma}

\section{Preliminaries}\label{prelim}
We are interested in parameters along the $k$th diagonal of staircase tableaux and in order to use Lemma~\ref{fac_mom}, we will need to compute asymptotically 
\[\mbP(x_{j_1}^{k}, x_{j_2}^k,\dots, x_{j_r}^k), \hspace{2mm} 1\le j_1<\dots<j_r\le n-k+1\] 
in the case where all of the $x$'s are $\a$'s (the case where all of the $x$'s are $\b$'s will follow by symmetry) and the case where any particular $x$ is allowed to be an $\a$ or a $\b$.

	To do this we consider two cases: $m=1$ and $m>1$ where
\[m:=\min\{l\ge1:\ j_l\le  j_{l+1}-k\}, \]
 i.e. all symbols on the $k^{th}$ diagonal before the $m^{th}$ are within $k$ columns of the subsequent symbol, but the $m^{th}$ is not. Consider arbitrary $S\in\mathcal{S}_n$ such that $x_{j_1}^{k}, x_{j_2}^k,\dots,x_{j_r}^k$. 
 By removing the first $j_1-1$ columns (see Equation~\eqref{subt}) we may and do assume that $j_1=1$ but it is important to note that $b$ now depends on $n$. 
  
  Consider its subtableau $\widehat S:=S[n-k-j_m+1,1]$, whose size is 
\[ \hk=n+1-(n-k-j_m+2)-(j_1-1)=k+j_m-1.\]

It has $m$ symbols on its $k^{th}$ diagonal and each such symbol forces two symbols on the first diagonal, namely an $\a$ in the same row and a $\b$ in the same column. In other words $x_{j_l}^k$ implies $\a_{k+j_l-1}^1$ and $\b_{j_l}^1$ (note that in this notation it is irrelevant if we consider $S$ or its subtableau $\widehat S$). Thus, $\widehat S$ contains at least $3m$ symbols, where at least $m$ of those symbols are $\a$'s and at least $m$ of those symbols are $\b$'s. 
  
Let $D$ be an area of $\widehat S$ bounded by its first diagonal and a broken line going vertically from $\b_1^1$ to $x_1^k$ then horizontally from $x_1^k$ to $0_{j_2}^{k+1-j_2}$ then vertically from $0_{j_2}^{k+1-j_2}$ to $x_{j_2}^k$ etc., until the vertical segment reaches $x_{j_m}^k$ and then the last horizontal segment goes from $x_{j_m}^k$ to $\a_{k+j_m-1}^1$ (see Figure~\ref{Amb} for an example). 

\begin{figure}[htbp] 
\setlength{\unitlength}{0.4cm}
\begin{center}
\begin{picture}
(0,0)(5,10)\thicklines


\put(0,0){\line(0,1){10}}
\put(1,0){\line(0,1){10}}
\put(2,1){\line(0,1){9}}
\put(3,2){\line(0,1){8}}
\put(4,3){\line(0,1){7}}
\put(5,4){\line(0,1){6}}
\put(6,5){\line(0,1){5}}
\put(7,6){\line(0,1){4}}
\put(8,7){\line(0,1){3}}
\put(9,8){\line(0,1){2}}
\put(10,9){\line(0,1){1}}

\put(0,10){\line(1,0){10}}
\put(0,9){\line(1,0){10}}
\put(0,8){\line(1,0){9}}
\put(0,7){\line(1,0){8}}
\put(0,6){\line(1,0){7}}
\put(0,5){\line(1,0){6}}
\put(0,4){\line(1,0){5}}
\put(0,3){\line(1,0){4}}
\put(0,2){\line(1,0){3}}
\put(0,1){\line(1,0){2}}
\put(0,0){\line(1,0){1}}

\put(0.25, 0.25){$\b$}
\put(0.25, 5.25){$x$}

\put(2.25, 7.25){$x$}
\put(2.25, 2.25){$\b$}


\put(4.25,9.25){$x$}
\put(4.25,4.25){$\b$}

\put(5.25, 5.25){$\a$}

\put(7.25,7.25){$\a$}

\put(9.25, 9.25){$\a$}

\put(0.5,0){\line(0,1){5.5}}
\put(2.5,5.5){\line(0,1){2}}
\put(4.5,7.5){\line(0,1){2}}

\put(0.5,5.5){\line(1,0){2}}
\put(2.5,7.5){\line(1,0){2}}
\put(4.5,9.5){\line(1,0){5.5}}

\end{picture}

\vspace{5cm}
\caption{A staircase tableau with $m=3$, $k=6$, $\hk=10$. 
\label{Amb}}

\end{center}

\end{figure}
Now for our methods, it is important to keep track of certain symbols in $\widehat S$ beyond those previously mentioned (i.e. beyond $x_{j_1}^{k}, x_{j_2}^k,\dots, x_{j_r}^k$ and the $\a$'s and $\b$'s on the first diagonal forced by those symbols). The following definitions and lemma will allow us to do so. Note that these symbols are precisely the $D$-connected symbols as will be defined shortly. 

\begin{definition}\label{abpath}
For any $x_j^i$, $i\neq1, i\neq k$, define its $\a\b$-path to be the path defined recursively by the following rules:
\begin{enumerate}
\item \label{betabelow} If $x_j^i=\a$ and $\exists l$, $1<l< i$, such that $\b_j^l$ (i.e. there is a $\b$ below in the same column), then connect those symbols. 
\item \label{alpharight} If $x_j^i=\b$ and $\exists h$, $i<h<i+j-h$, such that $\a_h^{i+j-h}$ (i.e. there is an $\a$ east in the same row), then connect those symbols.
\item \label{kdiagsymb} If $j\in\{j_1,\dots,j_m\}$ (or $j\in\{k-i-j_1,\dots,k-i-j_m\}$) (i.e. there is a symbol on the $k^{th}$ diagonal in the same column (or the same row)), then connect $x_j^i$ and $x_j^k$ (or $x^k_{k-i-j}$). 

\end{enumerate}
\end{definition}

\begin{definition}
Define $x_j^i$ to be $D$-connected if its $\a\b$-path ends on $D$ (See Figure~\ref{Dpaths}).
\end{definition}

\begin{lemma}\label{DcProp} Properties of $D$-connected symbols:
\begin{enumerate}
\item \label{DC1} Any symbol in the same row or the same column as a $D$-connected symbol is also $D$-connected.
\item \label{DC2} There are at most $\hk-m-1$ $D$-connected symbols in $\widehat S$.
\item \label{DC3} Each $D$-connected symbol can be paired uniquely with an opposite symbol on the diagonal.
\end{enumerate}
\end{lemma}

\begin{proof}
By Definition~\ref{abpath}(\ref{kdiagsymb}), this is trivial if the $D$-connected symbol is in the same row or column of a symbol on the $k^{th}$ diagonal. Otherwise, consider two cases. First, if the $D$-connected symbol is a $\b$, the only possible symbols in its row are $\a$'s which would be connected to the $\b$ by Definition~\ref{abpath}(\ref{alpharight}) and thus follow the same $\a\b$-path. The only potential symbols in the same column are $\b$'s and one $\a$ which is above all of the $\b$'s in this column (this $\a$ exists when not in the trivial case). All of the $\b$'s would be connected to this $\a$ and follow the same $\a\b$-path. The second case, if the $D$-connected symbol is an $\a$, follows similarly.

%

Since there are at most $2\hk-1$ symbols in $\widehat S$, and there are $m$ symbols on the $k^{th}$ and $\hk$ symbols on the first diagonal, there are at most $\hk-m-1$ D-connected symbols. 

The pairing can be defined as follows: pair each $D$-connected $\b$ with the $\a$ forced in the same row on the first diagonal. Pair each $D$-connected $\a$ with the $\b$ forced in the same column on the first diagonal. This pairing is unique since there is only one $\b$ in any given row and only one $\a$ in any given column. 
\end{proof}

\begin{figure}[htbp] 
\setlength{\unitlength}{0.4cm}
\begin{center}
\begin{picture}
(0,0)(4.5,11)\thicklines

\put(0,0){\line(0,1){11}}
\put(1,0){\line(0,1){11}}
\put(2,1){\line(0,1){10}}
\put(3,2){\line(0,1){9}}
\put(4,3){\line(0,1){8}}
\put(5,4){\line(0,1){7}}
\put(6,5){\line(0,1){6}}
\put(7,6){\line(0,1){5}}
\put(8,7){\line(0,1){4}}
\put(9,8){\line(0,1){3}}
\put(10,9){\line(0,1){2}}
\put(11,10){\line(0,1){1}}

\put(0,11){\line(1,0){11}}
\put(0,10){\line(1,0){11}}
\put(0,9){\line(1,0){10}}
\put(0,8){\line(1,0){9}}
\put(0,7){\line(1,0){8}}
\put(0,6){\line(1,0){7}}
\put(0,5){\line(1,0){6}}
\put(0,4){\line(1,0){5}}
\put(0,3){\line(1,0){4}}
\put(0,2){\line(1,0){3}}
\put(0,1){\line(1,0){2}}
\put(0,0){\line(1,0){1}}


\put(10.25, 10.25){$\a$}
\put(9.25, 9.25){$\b$}
\put(8.25, 8.25){$\a$}
\put(7.25, 7.25){$\a$}
\put(6.25, 6.25){$\b$}
\put(5.25, 5.25){$\a$}
\put(4.25, 4.25){$\a$}
\put(3.25, 3.25){$\b$}
\put(2.25, 2.25){$\b$}
\put(1.25, 1.25){$\b$}
\put(0.25, 0.25){$\b$}

\put(9.45, 10.25){$\a$}
\put(6.45, 7.25){$\a$}

\put(0.45,7.25){$\a$}

\put(2.4, 4.25){$\b$}
\put(2.45, 5.25){$\a$}
\put(0.4, 5.25){$\b$}
\put(3.45, 10.25){$\b$}

\put(.5, 0){\line(0,1){7.5}}
\put(0.5, 7.5){\line(1,0){3}}
\put(3.5, 7.5){\line(0,1){3}}
\put(3.5, 10.5){\line(1,0){7}}

\color{blue}
\put(0.04, 7.05){$\circ$}
\put(3.04, 10.05){$\circ$}

\put(0.04, 5.04){$\bullet$}
\put(0.25, 5.2){\line(0,1){1.95}}
\put(2.04, 4.04){$\bullet$}
\put(2.25, 4.2){\line(0,1){1}}
\put(9.04, 10.04){$\bullet$}
\put(3.35, 10.2){\line(1,0){6}}
\put(2.04, 5.04){$\bullet$}
\put(0.25, 5.2){\line(1,0){2}}

\put(6.04, 7.04){$\bullet$}
\put(0.35, 7.2){\line(1,0){6}}

\end{picture}

\vspace{5cm}
\caption{A staircase tableau of size $11$ with $m=2$, $k=8$, and $\hk=11$. $D$ is denoted by a black line. There are four $D$-connected symbols and their $\a\b$-paths are pictured with blue lines.
\label{Dpaths}}

\end{center}

\end{figure}
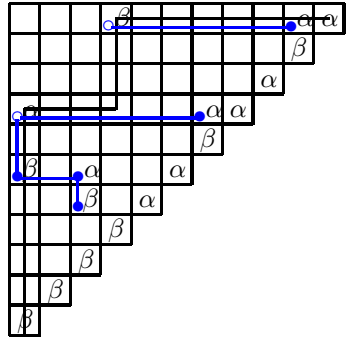

\section{The Distribution of $\a$'s}

In this section, we will consider the case where $x_j^k=\a$ for all $j\in\{j_1,j_2,\dots,j_r\}$. Notice that the $(1,1)$ and $(1,k)$ entries of $\widehat S$  are $\a$  and its $(k,1)$ entry is $\b$ (these correspond to the $(n-k+1,1)$, $(n-k+1,k)$, and $(n,1)$ entries of $S$). Similarly, the first diagonal entries in the same row as the $a_{j_2}^k,\dots,a_{j_m}^k$ are $\a$'s and in the same column are  $\b$'s. The remaining boxes in those columns and rows may have non--zero entries (namely, some of the boxes in those columns may have $\b$'s and some of the boxes in those rows may have $\a$'s). But notice that any of these symbols are $D$-connected. We let $D$ be the number of $D$-connected symbols and note that by Property~\eqref{DC2} from Lemma \ref{DcProp}, $1\leq D\leq \hk-m-1$.  
When adding the weights of the tableaux that contribute to the event $\{\a_1,a_{j_2}^k,\dots,\a_{j_r}^k\}$ we group the tableaux by the value of $D$. 
\[\sum_S wt(S)=\sum_{h=0}^{\hk-m-1}\sum_{S:\ D=h} wt(S),\]
where the sums are over those $S$ that satisfy $\a _{1}^k,\dots,\a_{j_r}^k$ (and in the second sum also have $h$ $D$-connected symbols).

If $D=h$, then by Property~(\ref{DC1}) and~(\ref{DC3}) from Lemma \ref{DcProp} the total contribution of weights from this case is, 
\[\a^{d}\b^{h+m}\sum_{T\in\mathcal S_{n-d}\atop \a_{j_{m-d+1}}^k,\dots,\a_{j_{r-d}}^k}wt(T).\]
where $d:=h+2m$ is used where convenient. Dividing by $Z_n(\a,\b)=(\a\b)^{d}(a+b+n-1)_{d}Z_{n-d}(\a,\b)$ gives that probability of this case to be
\[\frac {b^m}{(a+b+n-1)_{d}}\mbP_{n-d}(\a_{j_{m-d+1}}^k,\dots,\a_{j_{r-d}}^k).\] 

Define,
\[
C_{k,h}:=\#\text{$k\times k$ tableaux with $h$ $D$-connected symbols}.
\]

Note that $C_{k,h}$ depends on $k$ and $h$ but not on $n$ (and $C_{k,0}=1$). With this definition,  we would get
\[\mbP_n(\a_1^k,\a_{j_2}^k,\dots,\a_{j_r}^k,D=h)=C_{\hk,h} \frac {b^m}{(n+a+b-1)_{d}}\mbP_{n-d}(\a_{j_{m-d+1}}^k,\dots,\a_{j_{r-d}}^k),\]
In particular, it follows that 
\begin{eqnarray}
\mbP_n(\a_1^k,\a_{j_2}^k,\dots,\a_{j_r}^k)&=& \sum_{h=0}^{\hk-m-1}C_{\hk,h} \frac {b^m}{(n+a+b-1)_{h+2m}}\mbP_{n-d}(\a_{j_{m-d+1}}^k,\dots,\a_{j_{r-d}}^k)\label{breakdown}.
\end{eqnarray}

\subsection{The asymptotic distribution of $\a$'s on the k$^{th}$ diagonal}

We can now give the asymptotic joint distribution of the $\alpha$'s on the k$^{th}$ diagonal. Note that since we will be dealing exclusively with the k$^{th}$ diagonal through the remainder of this section we will drop the superscript $k$.
\begin{theorem}
\label{rasymbolsk}
Let $ 1 \leq j_{1} < ... < j_{r} \leq n-k+1$. If 
\begin{equation}\label{k_diff}j_{l} \leq j_{l + 1} - k, \hspace{2mm} \forall l = 1, 2, ..., r -
1\end{equation} 
then 
\[\mbP_{n,\a,\b}(\a_{j_1},...,\a_{j_r}) = \prod_{l=1}^{r}\frac{b + j_{r - l + 1} - 2r + 2l - 1}{(n+a+b-2r+2l-1)_2} + O\left(\frac{1}{(n+a+b)^{r+1}}\right). \]
Otherwise,
\[\mbP_{n,\a,\b}(\a_{j_1},...,\a_{j_r}) = O\left(\frac{1}{(n+a+b)^r}\right).
\]
\end{theorem}

\begin{proof} 
It suffices to prove the theorem for $j_1=1$. For if $j_1>1$, by Equation~\eqref{subtab}, we can consider $\mbP_{n-j_1+1,\a,\hb}(\a_{1},...,\a_{j_r-j_1+1})$ where $\hbb=b+j_1-1$ and the conclusion will follow from the $j_1=1$ case.


Therefore, W.L.O.G. assume $j_1=1$ and then the proof follows by induction on $r$. 
When $r=1$, apply Equation~\eqref{breakdown} (with $m=1$),
\beq*
\mbP_{n,\a,\b}(\a_{j_1})&=&
\sum_{h=0}^{k-2}C_{k,h} \frac {b}{(n+a+b-1)_{h+2}}\\
&=&
\frac {b}{(n+a+b-1)_{2}}+\sum_{h=1}^{k-2}C_{\hk,h} \frac {b}{(n+a+b-1)_{h+2}}.
\eeq*
Each of the expressions in the sum is $O(1/n^2)$ (recall that $b$ may depend on $n$). Therefore, the above is,
\[
=\frac b{(n+a+b-1)_{2}}+O(1/n^2)
\]
as required.

Assume the statement holds for integers up to $r-1$ and then consider the following two cases.

\begin{case}\label{first far} $m=1$.
\end{case}
Applying Equation\eqref{breakdown} with $m=1$,

\begin{eqnarray}\label{allcases}
&&\mbP_{n,\a,\b}(\a_{1},...,\a_{j_r})=\sum_{h=0}^{k-2}C_{k,h} \frac {b}{(n+a+b-1)_{h+2}}\mbP_{n-d}(\a_{j_{2-d}},\dots,\a_{j_{r-d}}).
\end{eqnarray}

The main contribution is when $h=0$. Recall that $C_{k,0}=1$.
\beq*
\mbP_{n,\a,\b}(\a_{1},...,\a_{j_r},h=0)=\frac {b}{(n+a+b-1)_{2}}\mbP_{n-2}(\a_{j_{2}-2},\dots,\a_{j_{r}-2}).
\eeq*

By the induction hypothesis, if $j_l\leq j_{l+1}-k$  for all $l=2,\dots,r-1$ then this is,
\beq*
&&\frac{b}{(n+a+b-1)_{2}}\left(\prod_{k=1}^{r-1}\frac{b + j_{r - k + 1}  - 2r + 2k-1 }{(n+a+b-2r+2k-1)_2}+O\left(\frac{1}{(n+a+b)^{r}}\right)\right)\\
&&\qquad=\prod_{k=1}^{r}\frac{b + j_{r - k + 1} - 2r + 2k-1}{(n+a+b-2r+2k-1)_2}+O\left(\frac{1}{(n+a+b)^{r+1}}\right).
\eeq*
On the other hand, if for some $2\le l\le r-1$, $j_l>j_{l+1}-k$ then by the induction hypothesis again and the fact that $b=O(n)$ the same expression is 
\beq*
&&\frac{b}{(n+a+b-1)_2}\cdot O\left(\frac{1}{(n+a+b)^{r-1}}\right)
=O\left(\frac{1}{(n+a+b)^r}\right).
\eeq*
Now, returning to Equation~\eqref{allcases}, when $h>0$, by the induction hypothesis,
\beq*
\mbP_{n,\a,\b}(\a_{1},...,\a_{j_r},h>0)
&=&\sum_{h=1}^{k-2}C_{k,h} \frac {b}{(n+a+b-1)_{h+2}}O\left(\frac{1}{(n+a+b)^{r-1}}\right)\\
&=&\sum_{h=1}^{k-2}C_{k,h}O\left(\frac{1}{(n+a+b)^{r+h}}\right)\\
&=&O\left(\frac{1}{(n+a+b)^{r+1}}\right).
\eeq*

\begin{case}\label{first close}  $m>1$. 
\end{case}
Notice that the cases $j_l=k$, $l\leq m$ are impossible by the rules of staircase tableaux and thus have probability zero.

For all other cases, by Equation~\eqref{breakdown}, 
 
\[
\mbP_{n,\a,\b}(\a_{j_1},...,\a_{j_r})
=\sum_{h=0}^{\hk-2m}C_{\hk,h} \frac {b^m}{(n+a+b-1)_{h+2m}}\mbP_{n-d}(\a_{j_{m-d+1}},\dots,\a_{j_{r-d}}).
\]
By the induction hypothesis,
\beq*
&=&\sum_{h=0}^{\hk-2m}C_{\hk,h} \frac {b^m}{(n+a+b-1)_{h+2m}}O\left(\frac{1}{(n+a+b)^{r-1}}\right)\\
&=&\sum_{h=0}^{\hk-2m}C_{\hk,h} O\left(\frac{1}{(n+a+b)^{r+h+m-1}}\right)\\
&=&O\left(\frac{1}{(n+a+b)^r}\right).\\
\eeq*

Combining Cases \ref{first far} and \ref{first close}, if $j_{l} \leq j_{l + 1} - k, \hspace{2mm} \forall\  l = 1, 2, \dots, r -1$, then 
\beq*
\mbP_{n,\a,\b}(\a_{j_1},...,\a_{j_r}) 
&=&\prod_{l=1}^{r}\frac{b + j_{r - l + 1} - 2r + 2l - 1}{(n+a+b-2r+2l-1)_2} + O\left(\frac{1}{(n+a+b)^{r+1}}\right).
\eeq*
Otherwise, 
\[
\mbP_{n,\a,\b}(\a_{j_1},...,\a_{j_r}) = 
O\left(\frac{1}{(n+a+b)^r}\right).
\]
\end{proof}

We can now give the asymptotic distribution of the number of $\a$'s on the $k$th diagonal.
\begin{theorem}\label{alpha_on_kth}
Let $A_{n}^{(k)}$ be the number of $\alpha$'s on the $k^{th}$ diagonal of a random weighted staircase tableau. Then,
as $n\to\infty$,
\[
A_{n}^{(k)} \stackrel{d}{\rightarrow} Pois \left( \frac{1}{2} \right).
\]
\end{theorem}

\begin{proof}
By \cite[Theorem~22, Chapter~1]{B} and Lemma~\ref{fac_mom}, it suffices to show that,
\[
\sum_{1 \leq j_{1} < ... < j_{r} \leq  n-k+1} \mathbb{P}(\a_{j_{1}},\ldots, \a_{j_{r}})\rightarrow \frac{1}{2^{r}r!}
\] 
as $n\rightarrow\infty$.
Write
\[
\sum_{1 \leq j_{1} < ... < j_{r} \leq  n-k+1} \mathbb{P}(\a_{j_{1}},\ldots, \a_{j_{r}})
=\sum_{J} \mathbb{P}(\a_{j_{1}},\ldots, \a_{j_{r}}) +
\sum_{J^c}\mathbb{P}(\a_{j_{1}},\ldots, \a_{j_{r}}), 
\]
where
\[J=\{1\le j_1<\dots<j_r\le  n-k+1:\  \forall\ l=1,\dots,r-1; j_{l+1}-j_l\geq k\}
\]
and
\[J^c=\{1\le j_1<\dots<j_r\le  n-k+1:\  \exists\ l=1,\dots,r-1; j_{l+1}-j_l<k\}.
\] 
By Theorem~\ref{rasymbolsk}, 
\beq* 
&&\sum_{J} \mathbb{P}(\a_{j_{1}},\ldots, \a_{j_{r}}) =
\sum_{J}\left(\prod_{l=1}^{r} \frac{b+j_{r-l+1}-2(r-l)-1}{(n+a+b-2(r-l)-1)_2}+O\left(\frac{1}{(n+a+b)^{r+1}}\right)\right)\\
&&\quad=\sum_{J}\left(\prod_{l=1}^{r} \frac{j_{l}}{n^2}+ O\left(\frac{1}{n^{2r}}\right)\right)+{n-k+1\choose r}\cdot O\left(\frac{1}{n^{r+1}}\right)
=\sum_{J_{r,n-k+1}}\prod_{l=1}^{r} \frac{j_{l}}{n^2}-\sum_{J_{r,n-k+1}\setminus J}\prod_{l=1}^{r} \frac{j_{l}}{n^2}+
O\left(\frac{1}{n}\right).
\eeq*
If $(j_1,\dots,j_r)\in J_{r,n-k+1}\setminus J$ then there exists an $l$ such that $j_{l+1}-j_l<k$ and thus this set has $O\left({n-k+1\choose r-1}\right)$ elements. Therefore, by Lemma~\ref{LA} the expression above is asymptotic to
\[\frac{1}{2^{r}r!}+O\left({n-k+1\choose r-1}\cdot \frac{n^r}{n^{2r}}\right)+O\left(\frac{1}{n}
\right)=
\frac{1}{2^{r}r!}+O\left(\frac{1}{n}\right),\quad\mbox{as\ }n \rightarrow \infty. 
\]
Finally, by Theorem~\ref{rasymbolsk},
\[\sum_{J^c}\mathbb{P}(\a_{j_{1}},\ldots, \a_{j_{r}})=O\left({n-k+1\choose r-1}\frac1{n^r}\right)=O\left(\frac1n\right).
\]
Combining these expressions and taking the limit as $n\rightarrow \infty$ completes the proof.
\end{proof}

\begin{corollary}\label{beta_on_kth}
Let $B_{n}^{(k)}$ to be the number of $\b$'s on the $k^{th}$ diagonal of a random weighted staircase tableau. Then,
as $n\to\infty$,
\[
B_{n}^{(k)} \stackrel{d}{\rightarrow} Pois \left( \frac{1}{2} \right).
\]
\end{corollary}
\begin{proof}
This follows from the involution discussed in Section~\ref{intro}.
\end{proof}


\section{The distribution of symbols}
In this section, we will consider the case where $x_j^k$ is any symbol for all $j\in\{j_1,j_2,\dots,j_r\}$. When adding the weights of the tableaux that contribute to the event $\{x_{j_1}^k,x_{j_2}^k,\dots,x_{j_r}^k\}$ we group the tableaux by the value of $D$. 
\[\sum_S wt(S)=\sum_{h=0}^{\hk-m-1}\sum_{S:\ D=h} wt(S),\]
where the sums are over those $S$ that satisfy $x _{j_1}^k,\dots,x_{j_r}^k$ (and in the second sum also have $h$ $D$-connected symbols).\\
Now to obtain an $x_{j_1}^{k}$ version of Equation~\eqref{breakdown}, we must handle the cases $m=1$ and $m>1$ separately.
\subsection{m=1}
When $x_{j_1}^{k}=\a$, the weight of $S$ follows easily from Equation~\eqref{breakdown},
\be
\mbP_n(\a_1,x_{j_2}^k,\dots,x_{j_r}^k, m=1)= \sum_{h=0}^{k-2}C_{k,h} \frac {b}{(n+a+b-1)_{h+2}}\mbP_{n-d}(x_{j_{2}-d}^k,\dots,x_{j_r-d}^k)\label{abreakdown}.
\ee
When $x_{j_1}^{k}=\b$, there is ambiguity in the boxes above this symbol and we must handle it carefully. Let $\b_j^i(0_l)$ be the event that there is a $\b$ in column $j$ and diagonal $i$ and there are $l$ empty boxes directly above $\b$ in the same column. 
\\
If $D=h$, then by Property (1) and (3) from Lemma \ref{DcProp} the total contribution of weights from this case is, 
\[\a^{h+1}\b^{h+1}\sum_{T\in\mathcal S_{n-h-1}\atop \b _{1}^1(0_{k-h-2}),x_{j_{2}-h-1}^k,\dots,x_{j_r-h-1}^k}wt(T).\]

Notice that the weight of a tableaux with $\b _{1}^1(0_{k-h-2})$ is the same as the weight of a tableaux with $\b _{k-h-1}^1$. Thus, the above is,

\[\a^{h+1}\b^{h+1}\sum_{T\in\mathcal S_{n-d+1}\atop \b _{k-d+1}^1,x_{j_{2}-d+1}^k,\dots,x_{j_r-d+1}^k}wt(T).\]

Dividing by $Z_n(\a,\b)=(\a\b)^{h+1}(a+b+n-1)_{h+1}Z_{n-h-1}(\a,\b)$ gives that probability of this case is

\beq*
&=&\frac {1}{(a+b+n-1)_{h+1}}\mbP_{n-h-1,\a,\b}(\b _{k-h-1}^1,x_{j_{2}-d+1}^k,\dots,x_{j_r-d+1}^k)\\
&=&\frac {1}{(a+b+n-1)_{h+1}}\left(\mbP_{n-h-1,\a,\b}(x_{j_{2}-d+1}^k,\dots,x_{j_r-d+1}^k)-\mbP_{n-h-1,\a,\b}(\a _{k-h-1}^1,x_{j_{2}-d+1}^k,\dots,x_{j_r-d+1}^k)\right).
\eeq*

Now, to compute the probability of the event $\{\a _{k-h-1}^1,x_{j_{2}}^k,\dots,x_{j_r}^k\}$, delete the first $k-h-2$ columns and let $\widetilde{b}=b+k-h-2$, $\widetilde{n}=n-k+h+2$, and $\widetilde{j}=j-k+h+2$. Then,

\beq*
\mbP_{n-h-1,\a,\b}(\a _{k-h-1}^1,x_{j_{2}-d+1}^k,\dots,x_{j_r-d+1}^k)&=&\mbP_{\widetilde{n}-h-1,\a,\widetilde{\b}}(\a _{1}^1,x_{\widetilde{j}_{2}-d+1}^k,\dots,x_{\widetilde{j}_r-d+1}^k)
\eeq*

Adding the weights of the tableaux that contribute to the event $\{\a _{1}^1,x_{\widetilde{j}_{2}-d+1}^k,\dots,x_{\widetilde{j}_r-d+1}^k\}$, we obtain

\[\a\sum_{T\in\mathcal S_{\widetilde{n}-d,\a,\widetilde{\b}}\atop x_{\widetilde{j}_{2}-d}^k,\dots,x_{\widetilde{j}_r-d}^k}wt(T).\]

Dividing by $Z_{\widetilde{n}-h-2}(\a,\widetilde{\b})=(\a\widetilde{\b})(a+\widetilde{b}+\widetilde{n}-h-3)Z_{\widetilde{n}-h-1}(\a,\widetilde{\b})$ gives that probability of this case is,
\[\frac{\widetilde{b}}{(a+\widetilde{b}+\widetilde{n}-h-3)}\mbP_{\widetilde{n}-h-2,\a,\widetilde{\b}}(x_{\widetilde{j}_{2}-d}^k,\dots,x_{\widetilde{j}_r-d}^k).\]

Therefore, the probability of $S$ when $D=h$ is,
\beq*
&=&\frac {1}{(a+b+n-1)_{h+1}}\left(\mbP_{n-h-1}(x_{j_{2}-d+1}^k,\dots,x_{j_r-d+1}^k)-\frac{\widetilde{b}}{(a+\widetilde{b}+\widetilde{n}-h-3)}\mbP_{\widetilde{n}-h-2,\a,\widetilde{\b}}(x_{\widetilde{j}_{2}-d}^k,\dots,x_{\widetilde{j}_r-d}^k)\right)\\
&=&\frac {1}{(a+b+n-1)_{h+1}}\left(\mbP_{n-h-1}(x_{j_{2}-d+1}^k,\dots,x_{j_r-d+1}^k)-\frac{b+k-h-2}{(a+b+n-h-3)}\mbP_{\widetilde{n}-h-2,\a,\widetilde{\b}}(x_{\widetilde{j}_{2}-d}^k,\dots,x_{\widetilde{j}_r-d}^k)\right).
\eeq*

In particular, it follows that 
\be
&&\mbP_n(\b_1,x_{j_2}^k,\dots,x_{j_r}^k, m=1)=\sum_{h=0}^{k-2}C_{k,h} \Big(\frac{1}{(a+b+n-1)_{h+1}}\mbP_{n-h-1}(x_{j_{2}-d+1}^k,\dots,x_{j_r-d+1}^k)\label{bbreakdown}\\
& &\quad-\frac{b+k-h-2}{(a+b+n-1)_{h+1}(a+b+n-h-3)}\mbP_{\widetilde{n}-h-2,\a,\widetilde{\b}}(x_{\widetilde{j}_{2}-d}^k,\dots,x_{\widetilde{j}_r-d}^k)\Big)\nonumber.
\ee

%
%
%

 \subsection{m$>$1} In this case, we will give an upper bound. Within $\{x^k_1,\dots, x^k_{j_m}\}$, say there are $l$ $\a$'s and $m-l$ $\b$'s for some $l$. Then if $D=h$ the total contribution of weights from this case is less than
\[\a^{2l}\b^{l}\a^{m-l}\b^{m-l}\a^{h}\b^{h}\sum_{T\in\mathcal S_{n-h-2l-(m-l)}\atop x_{j_{m+1}}^k,\dots,x_{j_r}^k}wt(T).\] 
This is an upper bound since the sum on the right-hand-side is less restricted than what we actually want to calculate. In particular, the $\b^k$'s force some empty boxes and a $\b$ below on the first diagonal, but we will consider instead all the extensions of those columns.\\
Dividing by $Z_n(\a,\b)=(\a\b)^{h+m+l}(a+b+n-1)_{h+m+l}Z_{n-h-m-l}(\a,\b)$ gives that probability of this case is
\[\frac {b^l}{(a+b+n-1)_{h+m+l}}\mbP_{n-h-m-l}(x_{j_{m+1}}^k,\dots,x_{j_r}^k).\]
(Notice that when $l=m$, we have equality as this is analogous to the $\a$ case).
Thus,  we would get
\[\mbP_n(x_1^k,x_{j_2}^k,\dots,x_{j_r}^k,D=h)\leq C_{\hk,h} \frac {b^l}{(n+a+b-1)_{h+m+l}}\mbP_{n-h-m-l}(x_{j_{m+1}}^k,\dots,x_{j_r}^k),\]
In particular, it follows that 
\be
\mbP_n(x_1^k,x_{j_2}^k,\dots,x_{j_r}^k, m>1)=O\left(\frac{1}{(n+a+b)^m}\right)\sum_{h=0}^{\hk-m-1}C_{\hk,h} \mbP_{n-h-m-l}(x_{j_{m+1}}^k,\dots,x_{j_r}^k)\label{breakdownm2}.
\ee
\subsection{The asymptotic distribution of symbols on the $k^{th}$ diagonal}
We can now give the asymptotic joint distribution of symbols on the $k$th diagonal. In fact, on the $k$th diagonal, symbols are asymptotically independent. To show this, we will first compute the probability of the event $\{x_{j_1},\dots,x_{j_{r_1+r_2}}\}$, where $0\leq r_1,r_2$ and $\{x_{j_1},\dots,x_{j_{r_1+r_2}}\}$ is a particular sequence of $\a$'s and $\b$'s. More specifically, we will say that there are $r_1$ $\a$'s and $r_2$ $\b$'s using the notation $\a_{j^a_1},\dots,\a_{j^a_{r_1}}$, $1\leq j^a_1<\dots<j^a_{r_1}\leq(n-k+1)$ and $\b_{j^b_1},\dots,\b_{j^b_{r_1}}$, $1\leq j^b_1<\dots<j^b_{r_2}\leq(n-k+1)$, for the ordering of $\a$'s and $\b$'s separately. As before, we will drop the superscript $k$.
 \begin{theorem}
\label{rsymbolsk} For all $1\leq j_1<\dots<j_{r_1+r_2}\leq (n-k+1)$. If,
\[
j_{l+1}-j_l\geq k, \forall l=1,\dots,r_1+r_2-1
\]
then
\beq*
\mbP_n(x_{j_1},\dots, x_{j_{r_1+r_2}})=\left(\prod_{l_1=1}^{r_1}\mbP_n(\a_{j^{a}_{r_1-l_1+1}})\right)\left(\prod_{l_2=1}^{r_2}\mbP_n(\b_{j^{b}_{r_2-l_2+1}})\right)+O\left(\frac{1}{(n+a+b)^{r_1+r_2+1}}\right)\\
=\left(\prod_{l_1=1}^{r_1}\frac{j^{a}_{r_1-l_1+1}+b-1}{(n+a+b-k+1)_2}\right)\left(\prod_{l_2=1}^{r_2}\frac{n-j^{b}_{r_2-l_2+1}-k+a+1}{(n+a+b-k+1)_2})\right)+O\left(\frac{1}{(n+a+b)^{r_1+r_2+1}}\right).
\eeq*
\\
Otherwise, 
\[
\mbP_n(x_{j_1},\dots, x_{j_{r_1+r_2}})=O\left(\frac{1}{(n+a+b)^{r_1+r_2}}\right).
\]
 \end{theorem}
 
 \begin{proof}
 As in Theorem~\ref{rasymbolsk}, it suffices to prove the theorem for $j_1=1$. The proof will be by induction on $r_1+r_2$. If $r_1+r_2=1$, then $j_1=j^a_1$ or $j_1=j_1^b$ and in either case the conclusion is trivial due to Equation~\eqref{boxk}. \\
 Assume the statement holds for integers up to $r_1+r_2-1$ and consider the following two cases:
 \begin{case}
$m=1$.
\end{case}
This case will be broken up into two possibilities:
 \begin{subcase}\label{fa}
 $j_1=j_1^a$
 \end{subcase}
 Applying Equation~\eqref{abreakdown}, 
 \begin{equation}
\mbP_n(\a_{1},x_{j_2},\dots,x_{j_{r_1+r_2}})=\sum_{h=0}^{k-2}C_{k,h} \frac {b}{(n+a+b-1)_{h+2}}\mbP_{n-d}(x_{j_2-d},\dots,x_{j_{r_1+r_2}-d})\label{parts}.
 \end{equation}
The main contribution is when $h=0$. Since $C_{k,0}=1$, \\
\beq*
\mbP_n(\a_{1},x_{j_2},\dots,x_{j_{r_1+r_2}},h=0)&=&\frac {b}{(n+a+b-1)_{2}}\mbP_{n-2}(x_{j_2-2},\dots,x_{j_{r_1+r_2}-2}).
\eeq*
By the induction hypothesis, if $j_l\leq j_{l+1}-k$ for all $l=1,\dots,r_1+r_2-1$ then this is
\be
 &&=\frac {b}{(n+a+b-1)_2}\left(\prod_{l_1=1}^{r_1-1}\frac{j^{a}_{r_1-l_1-1}+b-3}{(n+a+b-k-1)_2}\prod_{l_2=1}^{r_2}\frac{n-j^{b}_{r_2-l_2+1}-k+a+1}{(n+a+b-k+1)_2}+O\left(\frac{1}{(n+a+b)^{r_1+r_2}}\right)\right)\nonumber\\
 &&=\frac {b}{(n+a+b-1)_2}\prod_{l_1=1}^{r_1-1}\left(\frac{j^{a}_{r_1-l_1-1}+b-1}{(n+a+b-k-1)_2}-\frac{2}{(n+a+b-k-1)_2}\right)\nonumber\\
&&\quad\cdot\prod_{l_2=1}^{r_2}\frac{n-j^{b}_{r_2-l_2+1}-k+a+1}{(n+a+b-k+1)_2}+O\left(\frac{1}{(n+a+b)^{r_1+r_2+1}}\right)\nonumber\\ &&=\frac{b}{(n+a+b-1)_2}\left(\prod_{l_1=1}^{r_1-1}\frac{j^{a}_{r_1-l_1-1}+b-1}{(n+a+b-k-1)_2}+O\left(\frac{1}{(n+a+b)^{r_1}}\right)\right)\nonumber\\
&&\quad\cdot\prod_{l_2=1}^{r_2}\frac{n-j^{b}_{r_2-l_2+1}-k+a+1}{(n+a+b-k+1)_2}+O\left(\frac{1}{(n+a+b)^{r_1+r_2+1}}\right)\nonumber\\ &&=\frac{b}{(n+a+b-1)_2}\prod_{l_1=1}^{r_1-1}\frac{j^{a}_{r_1-l_1-1}+b-1}{(n+a+b-k-1)_2}\prod_{l_2=1}^{r_2}\frac{n-j^{b}_{r_2-l_2+1}-k+a+1}{(n+a+b-k+1)_2}+O\left(\frac{1}{(n+a+b)^{r_1+r_2+1}}\right)\label{PAEq}.
\ee
To complete the proof in this case, the denominators in each fraction above must be adjusted. As a demonstration, we will do so for the first fraction.
\be
\frac{b}{(n+a+b-1)_2}&=&\frac{b}{(n+a+b-k+1)_2}\frac{(n+a+b-k+1)_2}{(n+a+b-1)_2}\nonumber\\
&=&\frac{b}{(n+a+b-k+1)_2}\left(1+\frac{-k+2}{n+a+b-1}\right)\left(1+\frac{-k+2}{n+a+b-2}\right)\nonumber\\
&=&\frac {b}{(n+a+b-k+1)_2}+O\left(\frac{1}{(n+a+b)^2}\right)\label{demo}.
\ee
Therefore, returning to \eqref{PAEq},
\beq*
&=&\left(\frac {b}{(n+a+b-k+1)_2}+O\left(\frac{1}{(n+a+b)^2}\right)\right)\prod_{l_1=1}^{r_1-1}\left(\frac{j^{a}_{r_1-l_1-1}+b-1}{(n+a+b-k+1)_2}+O\left(\frac{1}{(n+a+b)^{r_1}}\right)\right)\\
&&\cdot\prod_{l_2=1}^{r_2}\left(\frac{n-j^{b}_{r_2-l_2+1}-k+a+1}{(n+a+b-k+1)_2}+O\left(\frac{1}{(n+a+b)^{r_2}}\right)\right)+O\left(\frac{1}{(n+a+b)^{r_1+r_2+1}}\right)\\
&=&\left(\prod_{l_1=1}^{r_1}\frac{j^{a}_{r_1-l_1+1}+b-1}{(n+a+b-k+1)_2}\right)\left(\prod_{l_2=1}^{r_2}\frac{n-j^{b}_{r_2-l_2+1}-k+a+1}{(n+a+b-k+1)_2})\right)+O\left(\frac{1}{(n+a+b)^{r_1+r_2+1}}\right).
\eeq*

On the other hand, if for some  $2\leq l\leq r-1, j_l> j_{l+1}-k$ then by the induction hypothesis again, this is
\beq*
 &=&\frac {b}{(n+a+b-1)_2}\cdot O\left(\frac{1}{(n+a+b)^{r-1}}\right).
\eeq*

Now returning to Equation~\eqref{parts}, when $h>0$, by the induction hypothesis,
\beq*
\mbP_n(\a_{1},x_{j_2},\dots,x_{j_{r_1+r_2}},h>0)&=&\sum_{h=1}^{k-2}C_{k,h} \frac {b}{(n+a+b-1)_{h+2}}O\left(\frac{1}{(n+a+b)^{r_1+r_2-1}}\right)\\
&=&O\left(\frac{1}{(n+a+b)^{r_1+r_2+1}}\right).
\eeq*

  \begin{subcase}
  $j_1=j^b_1$.
 \end{subcase}
 Applying Equation~\eqref{bbreakdown}, $\mbP_n(\b_1,\dots, x_{j_{r_1+r_2}})$ can be expressed as
\be
&&\sum_{h=0}^{k-2}C_{k,h} \Big(\frac{1}{(a+b+n-1)_{h+1}}\mbP_{n-d+1}(x_{j_2-d+1},\dots,x_{j_r-d+1})\nonumber\\
&&\quad-\frac{b+k-h-2}{(a+b+n-1)_{h+1}(a+b+n-h-3)}\mbP_{\widetilde{n}-d,\a,\widetilde{\b}}(x_{j_2-d},\dots,x_{j_r-d})\Big)\label{breakb}.
\ee
The main contribution is when $h=0$, \\
\be
\mbP_n(\b_1,\dots, x_{j_{r_1+r_2}},h=0)&=&\frac{1}{(a+b+n-1)}\mbP_{n-1}(x_{j_2-1},\dots,x_{j_r-1})\nonumber\\
&&-\frac{b+k-2}{(a+b+n-1)(a+b+n-3)}\mbP_{\widetilde{n}-2,\a,\widetilde{\b}}(x_{\widetilde{j}_2-2},\dots,x_{\widetilde{j}_r-2})\Big)\label{hiszero}.
\ee
By the induction hypothesis, if $j_l\leq j_{l+1}-k$ for all $l=1,\dots,r_1+r_2-1$ then this is
\beq*
&=&\frac{1}{(a+b+n-1)}\prod_{l_1=1}^{r_1}\frac{j^{a}_{r_1-l_1+1}+b-2}{(n+a+b-k)_2}\prod_{l_2=1}^{r_2-1}\frac{n-j^{b}_{r_2-l_2+1}-k+a}{(n+a+b-k)_2}+O\left(\frac{1}{(n+a+b)^{r_1+r_2+1}}\right)\nonumber\\
&&-\frac{b+k-2}{(a+b+n-1)(a+b+n-3)}
\prod_{l_1=1}^{r_1}\frac{j^{a}_{r_1-l_1+1}+b-3}{(n+a+b-k-1)_2}\prod_{l_2=1}^{r_2-1}\frac{n-j^{b}_{r_2-l_2+1}-k+a-1}{(n+a+b-k-1)_2}\\
&=&\frac{1}{(a+b+n-1)}\prod_{l_1=1}^{r_1}\frac{j^{a}_{r_1-l_1+1}+b-2}{(n+a+b-k)_2}\prod_{l_2=1}^{r_2-1}\frac{n-j^{b}_{r_2-l_2+1}-k+a}{(n+a+b-k)_2}+O\left(\frac{1}{(n+a+b)^{r_1+r_2+1}}\right)\nonumber\\
&&-\frac{b+k-2}{(a+b+n-1)(a+b+n-3)}
\prod_{l_1=1}^{r_1}\frac{j^{a}_{r_1-l_1+1}+b-3}{(n+a+b-k-1)_2}\prod_{l_2=1}^{r_2-1}\frac{n-j^{b}_{r_2-l_2+1}-k+a-1}{(n+a+b-k-1)_2}.
\eeq*

 
As in Case~\ref{fa}, we adjust the denominators as demonstrated in Equation~\eqref{demo}. In doing so, we obtain the following,

\beq*
&&\frac{1}{(a+b+n-1)}\left(\prod_{l_1=1}^{r_1}\frac{j^{a}_{r_1-l_1+1}+b-1}{(n+a+b-k+1)_2}+O\left(\frac{1}{(n+a+b)^{r_1}}\right)\right)\\
&&\cdot\left(\prod_{l_2=1}^{r_2-1}\frac{n-j^{b}_{r_2-l_2+1}-k+a+1}{(n+a+b-k+1)_2}+O\left(\frac{1}{(n+a+b)^{r_2-1}}\right)\right)+O\left(\frac{1}{(n+a+b)^{r_1+r_2+1}}\right)\\
&&-\frac{b+k-2}{(a+b+n-1)(a+b+n-3)}
\left(\prod_{l_1=1}^{r_1}\frac{j^{a}_{r_1-l_1+1}+b-1}{(n+a+b-k+1)_2}+O\left(\frac{1}{(n+a+b)^{r_1}}\right)\right)\\
&&\cdot\left(\prod_{l_2=1}^{r_2-1}\frac{n-j^{b}_{r_2-l_2+1}-k+a+1}{(n+a+b-k+1)_2}+O\left(\frac{1}{(n+a+b)^{r_2-1}}\right)\right)\\
&=& \frac{n+a-k-1}{(a+b+n-1)(a+b+n-3)}\left(\prod_{l_1=1}^{r_1}\frac{j^{a}_{r_1-l_1+1}+b-1}{(n+a+b-k+1)_2}+O\left(\frac{1}{(n+a+b)^{r_1}}\right)\right)\\
&&\cdot\left(\prod_{l_2=1}^{r_2-1}\frac{n-j^{b}_{r_2-l_2+1}-k+a+1}{(n+a+b-k+1)_2}+O\left(\frac{1}{(n+a+b)^{r_2-1}}\right)\right)+O\left(\frac{1}{(n+a+b)^{r_1+r_2+1}}\right)\\
&=&\prod_{l_1=1}^{r_1}\frac{j^{a}_{r_1-l_1+1}+b-1}{(n+a+b-k+1)_2}\prod_{l_2=1}^{r_2}\frac{n-j^{b}_{r_2-l_2+1}-k+a+1}{(n+a+b-k+1)_2}+O\left(\frac{1}{(n+a+b)^{r_1+r_2+1}}\right).
\eeq*

 
 On the other hand, if for some  $2\leq l\leq r-1, j_l> j_{l+1}-k$ then by the induction hypothesis again, Equation~\eqref{hiszero} becomes
 \beq*
&=& \frac{1}{a+b+n-1}O\left(\frac{1}{(n+a+b)^{r-1}}\right)-\frac{b+k-2}{(a+b+n-1)(a+b+n-3)}O\left(\frac{1}{(n+a+b)^{r-1}}\right)\nonumber\\
&=&O\left(\frac{1}{(n+a+b)^r}\right).
 \eeq*
 Now returning to Equation~\eqref{breakb}, when $h>0$, by the induction hypothesis,
\beq*
\mbP_n(\b_{1},x_{j_2},\dots,x_{j_{r_1+r_2}},h>0)&=&\sum_{h=1}^{k-2}C_{k,h}\Big(\frac{1}{(a+b+n-1)_{h+1}}O\left(\frac{1}{(n+a+b)^{r_1+r_2-1}}\right)\nonumber\\
&&-\frac{b+k-h-2}{(a+b+n-1)_{h+1}(a+b+n-h-3)}O\left(\frac{1}{(n+a+b)^{r_1+r_2-1}}\right)\Big)\\
&=&O\left(\frac{1}{(n+a+b)^{r_1+r_2+1}}\right).
\eeq*

\begin{case}
$m>1$.
\end{case}
By Equation~\eqref{breakdownm2},

\beq*
\mbP_n(x_{j_1},\dots, x_{j_{r_1+r_2}})&=& O\left(\frac{1}{(n+a+b)^m}\right)\sum_{h=0}^{\hk-m-1}C_{\hk,h} \mbP_{n-h-m-l}(x_{j_{m+1}}^k,\dots,x_{j_{r_1+r_2}}^k).
\eeq*

By the induction hypothesis,
\beq*
&=& O\left(\frac{1}{(n+a+b)^m}\right)\sum_{h=0}^{\hk-m-1}C_{\hk,h} O\left(\frac{1}{(n+a+b)^{r_1+r_2-m}}\right)\\
&=& O\left(\frac{1}{(n+a+b)^{r_1+r_2}}\right).
\eeq*

 \end{proof}
 
%
%
We can now give the asymptotic distribution of the number of symbols on the kth diagonal. As before, define $A_{n}$ and $B_{n}$ to be the number of $\alpha$'s and $\beta$'s respectively on the $k^{th}$ diagonal, 
i.e. 
$A_{n} := \sum^{n-k+2}_{j=1} I_{\alpha_j}$ and
$B_{n} := \sum^{n-k+2}_{j=1} I_{\beta_j}$. 
Also, let
\beq*
J=\{1\leq j_1^a<\dots<j_{r_1}^a\leq(n-k+1), 1\leq j_1^b<\dots<j_{r_2}^b\leq (n-k+1):\\
\forall l=1,\dots,r-1; j_{l+1}-j_l\geq k \}
\eeq*
and
\beq*
J^c=\{1\leq j_1^a<\dots<j_{r_1}^a\leq(n-k+1), 1\leq j_1^b<\dots<j_{r_2}^b\leq (n-k+1):\\
\exists l=1,\dots,r-1; j_{l+1}-j_l< k \}
\eeq*
Then we have the following,
\begin{theorem}\label{symbol_on_kth}
Let $Pois(\lambda)$ be a Poisson random variable with parameter
$\lambda$. Then, as $n\to\infty$,
\begin{equation}
\left(A_{n}, B_n\right) \stackrel{d}{\rightarrow} \left(Z_1,Z_2\right).
\end{equation}
where $Z_i\in Pois \left( \frac{1}{2} \right)$ are independent random variables.
\end{theorem}
\begin{proof}
By \cite[Theorem~23, Chapter~6]{B} it suffices to show that the joint
factorial moments of 
$(A_{n}, B_n)$ 
satisfy, for all $r_1,r_2\geq 0$,
\begin{equation}
\mathbb{E}[(A_{n})_{r_1}(B_{n})_{r_2}]\to\left(\frac12\right)^{r_1}\left(\frac12\right)^{r_2}\hspace{5mm}\mbox{as\ }n\to\infty.
\end{equation}

We have
\begin{eqnarray*} 
z_1^{A_n}z_2^{B_n} &=& z_1^{\sum^{n-k+2}_{j=1} I_{\alpha_j}}z_2^{\sum^{n-k+2}_{j=1} I_{\beta_j}} = \prod^{n-k+2}_{j=1}z_1^{I_{\alpha_j}}\prod^{n-k+2}_{j=1} z_2^{I_{\b_j}}=\prod^{n-k+2}_{j=1}(1 + (z_1-1))^{I_{\alpha_j}}\prod^{n-k+2}_{j=1} (1 + (z_2-1))^{I_{\b_j}} \\&=&
\prod^{n-k+2}_{j=1}(1 + I_{\alpha_j}(z_1-1))\prod^{n-k+2}_{j=1} (1 + I_{\b_j}(z_2-1))\\&=&
1 + \sum^{n-k+2}_{r_1=1} \sum^{n-k+2}_{r_2=1}\left( \sum_{J\cup J^{C}} \left( \prod^{r_1}_{l=1}I_{\alpha_j}\prod^{r_2}_{l=1} I_{\b_j}\right) \right) (z_1 - 1)^{r_1} (z_2 - 1)^{r_2}\\&=& 1 + \sum^{n-k+2}_{r_1=1} \sum^{n-k+2}_{r_2=1}(z_1 - 1)^{r_1} (z_2 - 1)^{r_2}\left( \sum_{J\cup J^{C}} \left( \prod^{r_1}_{l=1}I_{\alpha_j}\prod^{r_2}_{l=1} I_{\b_j}\right) \right).
\end{eqnarray*}

Thus, 
\beq*
\mathbb{E}(z_1^{A_n}z_2^{B_n})&=& 1 + \sum^{n-k+2}_{r_1=1} \sum^{n-k+2}_{r_2=1}(z_1 - 1)^{r_2}(z_1 - 1)^{r_2}\left( \sum_{J\cup J^{C}} \mbP_n(x_{j_1},\dots,x_{j_{r_1+r_2}} ) \right),
\eeq*

and,

\be
\mathbb{E}[(A_n)_r(B_n)_r]&=& \frac{d^{r_1}}{dz_1^{r_1}}\frac{d^{r_2}}{dz_2^{r_2}}(\mathbb{E}(z^{A_nB_n}))|_{z_1=z_2=1}=r_1!r_2!\left( \sum_{J\cup J^{C}} \mbP_n(x_{j_1},\dots,x_{j_{r_1+r_2}} ) \right)\label{ALLT}.
\ee

To compute the sum on the right-hand-side, first sum over $J$. By Theorem \ref{rsymbolsk},
\beq*
&=& \left( \sum_{J} \left( \left(\prod_{l=1}^{r_1}\frac{j_{\a_l}+b-1}{(n+a+b-k+1)_2}\right)\left(\prod_{l=1}^{r_2}\frac{n-j_{\b_l}-k+a+1}{(n+a+b-k+1)_2})\right)+O\left(\frac{1}{(n+a+b)^{r+1}}\right)\right)\right) \\
&\approx&\left( \sum_{J} \left( \left(\prod_{l=1}^{r_1}\frac{j_{\a_l}}{n^2}\right)\left(\prod_{l=1}^{r_2}\frac{n-j_{\b_l}}{n^2}\right)+O\left(\frac{1}{(n+a+b)^{r+1}}\right)\right)\right)\\
&=&\left( \left(\sum_{J_{r_1, n-k+2}}\prod_{l=1}^{r_1}\frac{j_{\a_l}}{n^2}\right)\left(\sum_{J_{r_2, n-k+2}}\prod_{l=1}^{r_2}\frac{n-j_{\b_l}}{n^2}\right)+O\left(\frac{1}{(n+a+b)^{r+1}}\right)\right)\\
&=&\left( \left(\sum_{J_{r_1, n-k+2}}\prod_{l=1}^{r_1}\frac{j_{\a_l}}{n^2}\right)\left(\sum_{k-1\leq j_1^b<\dots<j_{r_2}^b\leq n-1}\prod_{l=1}^{r_2}\frac{j_{\b_l}}{n^2}\right)+O\left(\frac{1}{(n+a+b)^{r+1}}\right)\right)\\
&=&\left( \left(\sum_{J_{r_1, n-k+2}}\prod_{l=1}^{r_1}\frac{j_{\a_l}}{n^2}\right)\left(\sum_{1\leq j_1^b<\dots<j_{r_2}^b\leq n-1}\prod_{l=1}^{r_2}\frac{j_{\b_l}}{n^2}-\sum_{1\leq j_1^b<\dots<j_{r_2}^b\leq k-1}\prod_{l=1}^{r_2}\frac{j_{\b_l}}{n^2}\right)+O\left(\frac{1}{(n+a+b)^{r+1}}\right)\right).\\
\eeq*
By Lemma~\ref{LA},
\beq*
&=& \left(\left(\frac{1}{n^{2r_1}}\frac{(n-k)_{2r_1}}{2^{r_1}r_1!}\right)\left(\frac{1}{n^{2r_2}}\frac{(n)_{2r_2}}{2^{r_2}r_2!}-\frac{1}{n^{2r_2}}\frac{(k)_{2r_2}}{2^{r_2} r_1!}\right)+O\left(\frac{1}{(n+a+b)^{r+1}}\right)\right)\\
&\approx& \left(\left(\frac{1}{2^{r_1}r_1!}\right)\left(\frac{1}{2^{r_2}r_2!}\right)+O\left(\frac{1}{(n+a+b)^{r+1}}\right)\right)\\
&=& \frac{1}{2^{r_1+r_2}r_1!r_2!}+O\left(\frac{1}{(n+a+b)^{r+1}}\right)
\eeq*

If $(j_1,\dots,j_{r_1+r_2})\in J^c$, then there exists an $l$ such that $j_{l+1}-j_l<k$ and thus this set has $O\left(\binom{n-k+1}{r_1+r_2-1}\right)$ elements. Therefore, summing over $J^{c}$, 

\beq*
\left( \sum_{J^{C}} \mbP_n(x_{j_{1}},\dots, x_{j_{r_1+r_2}}) \right)=O\left(\binom{n-k+1}{r_1+r_2-1}\frac{1}{n^{r_1+r_2}}\right)=O\left(\frac{1}{n}\right).
\eeq*

Returning to Equation~\eqref{ALLT} and letting $n\rightarrow \infty$ completes the proof.
\end{proof}

\end{document}